\documentclass[runningheads]{llncs}

\usepackage[T2A,T1]{fontenc}
\usepackage{amssymb, amsmath}
\usepackage{cleveref}
\usepackage{ifthen}

\let \leq \leqslant
\let \nleq \nleqslant
\let \geq \geqslant
\let \ngeq \ngeqslant
\let \phi \varphi
\let \epsilon \varepsilon
\let \embeds \hookrightarrow

\renewcommand \L [2][]{L_{\ifthenelse{\equal{#1}{}}\ngeq{\ngeq,#1}}(#2)}
\renewcommand \l [2][]{\o(\L[#1]{#2})}
\newcommand \K [2][]{L_{\ifthenelse{\equal{#1}{}}{\bot}{\bot,#1}}(#2)}
\renewcommand \k [2][]{\w(\K[#1]{#2})}

\renewcommand \S [1][]{\mathsf{P_f^{#1}}}
\def \T {\mathcal{T}}
\def \F {\mathcal{F}}

\def \h {\textbf{h}}
\def \o {\textbf{o}}
\def \w {\textbf{w}}

\title{Maximal Order Type and Width of the $h$-Preorder on Finite Labeled Trees}
\author{Ilya Smirnov\thanks{The work of the author was supported by the Russian Science Foundation, project 23-11-00133-\mbox{\usefont{T2A}{\rmdefault}{m}{n}\CYRP}.}}
\authorrunning{I. Smirnov}
\institute{Department of Mathematics and Computer Science\\ St. Petersburg University, 199178 St. Petersburg, Russia}

\begin{document}

\maketitle

\begin{abstract}
    Important measures of a well-quasi-order are its ordinal invariants, namely the height, the maximal order type, and the width. In this work, we continue the research on the ordinal invariants of finite labeled trees and generalize the results presented in \cite{Selivanov2025}.
    We establish tight bounds for the width and maximal order type of the $h$-preorder on $Q$-labeled trees in terms of the corresponding invariants of Q. Specifically, we prove that for a well-quasi-order $Q$, the width and the maximal order type satisfy $\w(\T(Q)) \geq \phi_{\w(Q)}(0)$ and $\o(\T(Q)) \leq \phi_{\o(Q)}(0)$, with equality when $\o(Q) = \w(Q)$.
\end{abstract}

\section{Introduction}

Well-quasi-orders (which we abbreviate as WQOs) are a particular case of the well-founded preorders, which have no infinite antichains. They have applications in many branches of mathematics, logic and computer science, including, but not limited to, proof theory, reverse mathematics, combinatorics, graph theory, and program verification.

In the finite case, there are several measures of the preorder: its size, the length of a maximal chain, and the length of a maximal antichain. They could be generalized to any WQO as \textit{ordinal invariants}: the \textit{height} (maximal order type of a chain), the \textit{maximal order type} of linearisation, and the \textit{width} (which is harder to describe; informally, it measures how antichains relate to each other). We give a formal uniform definition of these invariants in \cref{wqos}.

Among the most important examples of WQOs are tree orderings. Given a WQO $Q$, one can consider the set of $Q$-labeled trees ordered by homeomorphic embeddability. Kruskal proved that this is also a WQO \cite{Kruskal1960}.

In this paper, we focus on a weaker order on the same set: the $h$-preorder (we define it in \cref{trees}). This preorder was introduced in \cite{Selivanov2011} and it is closely connected with the Wadge reducibility and hierarchies of $k$-partitions \cite{Hertling1993,Selivanov2011}. While it might look similar to Kruskal's ordering, their ordinal invariants are very different.

The maximal order type of Kruskal's ordering was estimated in \cite{Schmidt2020}. In the previous work \cite{Selivanov2025} by the present author and Victor Selivanov, the maximal order type and width for the $h$-preorder over a finite antichain were evaluated (also, it is easy to see that the height in this case is equal to $\omega$). In the forthcoming paper by the author and Selivanov, the height of $h$-preorder was investigated. In this work, we apply the ideas from \cite{Selivanov2025,Friedman2023} to estimate the maximal order type and the width of the $h$-preorder over any WQO.

In \cref{section o}, we prove that $\o(Q) \leq \o(\T(Q)) \leq \phi_{-1+\o(Q)}(0)$, where $\phi$ denotes the Veblen function (the definition is given in \cref{arith}), and $\o$ denotes the maximal order type. In \cref{section w}, we show that $\w(\T(Q))$ could be arbitrarily large when the width $\w(Q)$ is fixed, and prove lower bounds $\w(\T(Q) \sqcup 1) \geq \phi_{\o(Q)}(0)$ and $\w(\T(Q)) \geq \phi_{-1+\w(Q)}(0)$. Together, these estimates show that if $\o(Q) = \w(Q) = \alpha$ then $\o(\T(Q)) = \w(\T(Q)) = \phi_{-1+\alpha}(0)$, hence the bounds are tight.

\section{Preliminaries}

\subsection{Ordinal arithmetic}\label{arith}

We assume that the reader is familiar with ordinal addition, multiplication, and exponentiation.

For an ordinal $\alpha \geq 1$, we write $-1+\alpha$ to denote the unique ordinal $\beta$ such that $1+\beta = \alpha$. For finite $\alpha$, it is $\alpha-1$; for infinite $\alpha$, it is $\alpha$.

Any positive ordinal $\alpha$ has a unique representation in \textit{Cantor normal form} $\alpha = \omega^{\alpha_1}+\dots+\omega^{\alpha_n}$ where $\alpha_1 \geq \dots \geq \alpha_n$ and $1 \leq n < \omega$.

Let $\alpha = \omega^{\alpha_1}+\dots+\omega^{\alpha_n}$ and $\beta = \omega^{\beta_1}+\dots+\omega^{\beta_n}$ be ordinals in Cantor normal forms. The \textit{natural sum} $\alpha \oplus \beta$ is defined as the sum of $\omega^{\alpha_1},\dots,\omega^{\alpha_n},\omega^{\beta_1},\dots,\omega^{\beta_m}$ in non-increasing order. The \textit{natural product} $\alpha \otimes \beta$ is defined as $\bigoplus_{i<n,j<m} \omega^{\alpha_i\oplus\beta_j}$. In contrast to regular ordinal sum and product, natural operations are commutative.

The \textit{Veblen function} $\phi_\alpha(\beta)$ is defined by recursion on $\alpha$ as follows: $\phi_0(\beta) = \omega^\beta$, and for $\alpha > 0$, $\phi_\alpha$ enumerates in increasing order the ordinals that are fixed points of $\phi_{\alpha'}$ for all $\alpha' < \alpha$. The \textit{fixed point free Veblen function} $\overline{\phi}_\alpha(\beta)$ is defined as:
\begin{itemize}
    \item $\overline{\phi}_\alpha(\beta) = \phi_\alpha(\beta+1)$ if $\beta = \gamma+n$ for some $n < \omega$ and $\phi_\alpha(\gamma) = \gamma$.
    \item $\overline{\phi}_\alpha(\beta) = \phi_\alpha(\beta)$ otherwise.
\end{itemize}

An ordinal is called an \textit{epsilon number} if it is in the image of $\phi_1$, i.e., it satisfies the equation $\omega^x = x$. An ordinal $\alpha$ is \textit{additively (resp. multiplicatively) indecomposable}, if it is closed under addition (resp. multiplication). $\omega$ and all epsilon numbers are additively and multiplicatively indecomposable.

\subsection{Well-quasi-orders}\label{wqos}

We abuse the notation and simply write $Q$ instead of $(Q, \leq)$. Given a subset of the preorder, we implicitly assume that it is equipped with an induced preorder. For an element $x \in X$ and a binary relation $R$ on $X$, we define $L_{R,X}(x) = \{y \in X \mid y R x\}$. We often omit the index $X$ and write $L_R(x)$.

For elements $x, y$ of preorder $Q$, we write $x \perp y$ to mean that $x$ and $y$ are incomparable, i.e. $x \not\leq y$ and $y \not\leq x$.

A sequence of elements of a preorder $Q$ is called \textit{bad} if $x_i \not\leq x_j$ for every pair of indices $i < j$. A preorder is a \textit{well-quasi-order} (WQO) if it does not contain infinite bad sequences. Well-quasi-orders have many equivalent characterizations. We mention two of them:
\begin{itemize}
    \item A preorder $Q$ is a WQO if and only if it contains neither infinite descending chains nor infinite antichains.
    \item A preorder $Q$ is a WQO if and only if every union $\bigcup_{i \in I} U_i$ of its upwards-closed subsets is equal to some finite union of them.
\end{itemize}

Given a WQO $Q$, let $Dec(Q)$, $Bad(Q)$, and $Ant(Q)$ be the trees of decreasing sequences, bad sequences, and antichains of $Q$, respectively, ordered by inverse prefix order. Since $Q$ does not contain infinite decreasing sequences (resp., bad sequences, antichains), these trees are well-founded, and their \textit{ranks} can be defined. The \textit{rank} of a node is the least ordinal that is strictly larger than the ranks of its smaller nodes. The \textit{rank} of a tree is the rank of its root (which, in our case, is the empty sequence). Now, the \textit{height} $\h(Q)$, \textit{maximal order type} $\o(Q)$, and \textit{width} $\w(Q)$ of $Q$ are defined as ranks of these trees, respectively.

Since every decreasing sequence or antichain is a bad sequence, we have $\h(Q), \w(Q) \leq \o(Q)$. These invariants can be computed inductively through the \textit{residual equations}:
\begin{align*}
    \h(Q) = \sup_{x \in Q} (\h(L_<(x)) + 1)\\
    \o(Q) = \sup_{x \in Q} (\l{x} + 1)\\
    \w(Q) = \sup_{x \in Q} (\k{x} + 1)
\end{align*}

Let $P$ and $Q$ be preorders. A \textit{reflection} is a function $f \colon P \to Q$ such that for all $x, y \in P$, if $f(x) \leq f(y)$, then $x \leq y$. If there exists a reflection from $P$ to $Q$, we write $P \embeds Q$.

If $f$ is a reflection, then it maps bad sequences to bad sequences and antichains to antichains, so we obtain the following lemma:
\begin{lemma}\label{embedding}
    Let $Q$ be a WQO and suppose that $P \embeds Q$. Then $P$ is also a WQO and $\o(P) \leq \o(Q)$, $\w(P) \leq \w(Q)$.
\end{lemma}

\subsection{Operations on WQOs}

The class of WQOs is closed under many operations: disjoint sum, lexicographic sum, direct and cartesian products, etc. Here we cite results about the operations that are used in this article. For a broader survey of these and other constructions and their ordinal invariants, see, e.g., \cite{Dzamonja2020}.

For preorders $(P, \leq_P)$ and $(Q, \leq_Q)$, the \textit{disjoint sum} $P \sqcup Q$ is the disjoint union of $P$ and $Q$ with the order $\leq_P \cup \leq_Q$, and the \textit{lexicographic sum} $P + Q$ is the disjoint union of $P$ and $Q$ with the order $\leq_P \cup \leq_Q \cup P \times Q$.

\begin{lemma}[\cite{Abraham1999,deJongh1977,Dzamonja2020}]\label{h o w(cup)}
    If $P$ and $Q$ are WQOs, then so are $P \sqcup Q$ and $P + Q$. The ordinal invariants are computed as:
    \begin{align*}
        &\h(P \sqcup Q)=\max(\h(P),\h(Q))&&\h(P+Q)=\h(P)+\h(Q)\\
        &\o(P \sqcup Q)=\o(P)\oplus\o(Q)&&\o(P+Q)=\o(P)+\o(Q)\\
        &\w(P \sqcup Q)=\w(P)\oplus\w(Q)&&\w(P+Q)=\max(\w(P),\w(Q))
    \end{align*}
\end{lemma}

In the \textit{cartesian product} $P \times Q$ of preorders $P$ and $Q$, the elements are compared component-wise.

\begin{lemma}[\cite{deJongh1977}]\label{h o(times)}
    If $P$ and $Q$ are WQOs, then so is $P \times Q$. The height and maximal order type are computed as:
    \begin{align*}
        &\h(P \times Q) = \sup_{\alpha < \h(P), \beta < \h(Q)} \alpha\oplus\beta+1\\
        &\o(P \times Q) = \o(P) \otimes \o(Q)
    \end{align*}
\end{lemma}

For a preorder $Q$, the \textit{finitary powerset} $\S(Q)$ is the set of finite subsets of $Q$ with the \textit{domination preorder}: $X \leq_{\S(Q)} Y \iff \forall a \in X \exists b \in Y (a \leq b)$. It is well-known that if $Q$ is a WQO, then $\S(Q)$ is also a WQO (note that the full powerset need not be, as shown in \cite{Rado1954}). A comprehensive survey on ordinal invariants of this structure is given in \cite{Abriola2023}.

\begin{lemma}[{\cite[theorem 3.3]{Abriola2023}}]\label{o(S)}
    For a WQO $Q$, $1+\o(Q) \leq \o(\S(Q)) \leq 2^{\o(Q)}$
\end{lemma}

\begin{lemma}[{\cite[theorem 3.11]{Abriola2023}}]\label{w(S)}
    Let $Q$ be a WQO. If $\w(Q)$ is infinite, then $\w(\S(Q)) \geq 2^{\w(Q)}$. If $\w(Q) = n < \omega$, then $\w(\S(Q)) \geq \binom{n}{\lfloor n/2 \rfloor}$.
\end{lemma}

For a preorder $Q$, let $\S[2](Q) \subset \S(Q)$ be the subset of those non-empty sets that are not equivalent to any singleton set. If $Q$ is a WQO, then $\S[2](Q)$ is a WQO as a subset of a WQO. The proof from \cite[theorem 3.11]{Abriola2023} can easily be modified for $\S[2]$.

\begin{lemma}\label{w(S2)}
    Let $Q$ be a WQO. If $\w(Q)$ is infinite then $\w(\S[2](Q)) \geq 2^{\w(Q)}$. If $3 \leq \w(Q) = n < \omega$, then $\w(\S[2](Q)) \geq \binom{n}{\lfloor n/2 \rfloor}$.
\end{lemma}

\subsection{Labeled trees}\label{trees}

A \textit{(finite) forest} is a finite ordered set $(F, \leq_F)$ such that for any $v \in F$, the subset $L_{<_F}(v)$ is linearly ordered. The elements of $F$ are called \textit{vertices}. If $u <_F v$, then $u$ is called an \textit{ascendant} of $v$, and $v$ is called a \textit{descendant} of $u$. A \textit{leaf} is a vertex that does not have any descendants.

A \textit{tree} is a forest that has a unique minimal element. A \textit{subtree} of a forest $(F, \leq_F)$ is any subset $T \subset F$ that is a tree (by induced order). A \textit{subtree of the vertex $v$} is the subtree $L_{\geq_F}(v)$.

For a set $Q$, a \textit{(finite) $Q$-labeled forest} (resp. \textit{tree}) is a forest (resp. tree) $F$ with a \textit{labeling function} $f \colon F \to Q$.

\begin{definition}
    For a preorder $Q$, the $h$-preorder on $Q$-labeled forests is defined by: $(F, f) \leq_h (G, g)$ if there exists a monotone function $h \colon F \to G$ such that $f(v) \leq g(h(v))$ for each vertex $v \in F$. The factor poset of $Q$-labeled forests (resp. trees) with $h$-preorder is denoted by $\F(Q)$ (resp. $\T(Q)$).
\end{definition}

This preorder, introduced in \cite{Selivanov2011}, is a variant of \textit{Kruskal's tree ordering} \cite{Kruskal1960}. Kruskal's ordering imposes the additional constraint that $h$ is strictly monotone and preserves infimums. We will denote by $\T_K(Q)$ the factor poset of $Q$-labeled trees with Kruskal's ordering. Kruskal's tree theorem states that if $Q$ is a WQO, then $\T_K(Q)$ is a WQO.

Another variant of tree ordering (which was considered in \cite{Montalban2017,Friedman2023}) is obtained if we require the function $h$ to be strictly monotone (but it does not have to preserve infimums); the resulting preorder we denote by $\T_s(Q)$. Obviously, the identity maps are reflections $\T(Q) \embeds \T_s(Q) \embeds \T_K(Q)$; hence if $Q$ is a WQO, then $\T(Q)$ and $\T_s(Q)$ are WQOs.

There have been several studies on ordinal invariants of these structures. In \cite{Schmidt2020}, the upper bound on $\o(\T_K(Q))$ is given. \cite{Friedman2023} proves the lower and upper bounds on $\o(\T_s(Q))$. The maximal order type and the width of the leaf-labeled variant of $\T_s$ are evaluated in \cite{Chopra2026}.

In the previous work \cite{Selivanov2025} (joint with Selivanov), we showed that the maximal order type and the width of $\T\left(\overline{k}\right)$ and $\F\left(\overline{k}\right)$ are equal to $\phi_{k-1}(0)$, where $\overline{k}$ is an antichain of size $k$. In this work, we generalize these results to an arbitrary WQO $Q$. The height of $\T(Q)$ is the subject of the forthcoming paper.

Now, we give some facts about the structure of trees and forests. We identify a labeled forest with its underlying set. Each $Q$-labeled forest $F$ has a unique representation $F \simeq t_1 \sqcup \dots \sqcup t_n$ where $t_1,\dots,t_n$ are pairwise incomparable $Q$-labeled trees, up to reordering and equivalence of $t_i$. Given a label $r \in Q$ and a forest $F \in \F(Q)$, let $rF$ be the tree with the underlying set $\{*\}+F$ and the label $r$ on $*$. It is easy to see that $h$-preorder satisfies the following recurrence:

\begin{lemma}
    \begin{itemize}
        \item $t_1 \sqcup \dots \sqcup t_n \leq F$ if and only if $t_i \leq F$ for all $i$.
        \item For a tree $t$, $t \leq t_1 \sqcup \dots \sqcup t_n$ if and only if $t \leq t_i$ for some $i$.
        \item For labels $r \leq s$, $r F \leq s G$ if and only if $F \leq s G$.
        \item For labels $r \nleq s$, $r F \leq s G$ if and only if $r F \leq G$.
    \end{itemize}
\end{lemma}

The first two points imply that $\F(Q)$ is isomorphic to $\S(\T(Q))$. Because of that, we identify the forest $t_1 \sqcup \dots \sqcup t_n$ with the set of trees $\{t_1, \dots, t_n\}$.

A $Q$-labeled forest $t$ is \textit{minimal}, if it is not $h$-equivalent to a $Q$-labeled forest with a lesser number of vertices. Any forest is equivalent to unique minimal forest, up to isomorphism of tree and equivalence of labels. We cite characterization of minimal forests from \cite{Selivanov2011}.
\begin{itemize}
    \item A tree $rF$ with $|F| \neq 1$ is minimal if and only if $F$ is minimal, and for any tree $t$ from $F$, the root of $t$ is not less than or equal to $r$.
    \item A tree $rt$ (where $t$ is a tree) is minimal if and only if $t$ is minimal, and the root of $t$ and $r$ are incomparable.
    \item A forest $F$ is minimal if and only if all its trees are minimal and pairwise incomparable.
\end{itemize}

In the rest of the article we implicitly assume that trees are minimal. In particular, the root of the tree is the root of its minimal representation (defined up to equivalence). Note that the root of $rf$ is not necessarily equal to $r$: for example, if $r < s$ are two labels, then the root of $rs$ is $s$, since $rs \simeq s$.

\section{Maximal order type}\label{section o}

First, we prove the lower bound for the maximal order type and demonstrate its tightness.
The trivial embedding $Q \embeds \T(Q)$ gives $\o(\T(Q)) \geq \o(Q)$; if $Q$ is an ordinal, then $\T(Q) \cong Q$; hence $\o(\T(Q)) = \o(Q)$, and the bound is tight. Since $\F(Q) \cong \S(\T(Q))$, by \cref{o(S)}, we have $\o(\F(Q)) \geq 2^{\o(Q)}$, which is also tight. In the rest of this section, we prove the upper bound on the maximal order type.

Since $\T(Q) \embeds \T_s(Q)$, the upper bound on $\o(\T(Q))$ can be obtained via an analogous bound for $\T_s(Q)$. Because of that, we cite the next lemma from \cite{Friedman2023}.

\begin{definition}
    For WQOs $I$ and $J$, let $\T_s(I, J)$ denote the subset of $\T_s(I \sqcup J)$ of those labeled trees whose leaves have labels from $J$ and non-leaves have labels from $I$.
\end{definition}

\begin{lemma}[{\cite[lemma 6.4]{Friedman2023}}]\label{o(T(I,J))}
    If $I,J$ are WQOs, then $\o(\T_s(I, J)) \leq \overline{\phi}_{\o(I)}(\o(J))$.
\end{lemma}

If we reflect $\T(Q) \embeds \T_s(Q, Q)$ by the identity map, we get $\o(\T(Q)) \leq \o(\T_s(Q, Q)) \leq \overline{\phi}_{\o(Q)}(\o(Q))$, which is quite loose. If instead we reflect $\T(Q) \embeds \T_s(Q, \{*\})$ by appending a $*$-labeled leaf to each vertex, we get $\o(\T(Q)) \leq \o(\T(Q, \{*\})) \leq \overline{\phi}_{\o(Q)}(1)$, which is better but remains insufficient. Thus, greater accuracy is required to obtain a precise bound.

\begin{theorem}
    For a non-empty WQO $Q$, $\o(\T(Q)) \leq \phi_{-1 + \o(Q)}(0)$.
\end{theorem}
\begin{proof}
    The case $\o(Q) = 1$ is trivial. If $2 \leq \o(Q) < \omega$, then $\T(Q) \embeds \T\left(\overline{\o(Q)}\right)$, and $\o(\T(Q)) \leq \o\left(\T\left(\overline{\o(Q)}\right)\right) \leq \phi_{\o(Q)-1}(0)$, where the last inequality is the statement of \cite[proposition 1]{Selivanov2025}. Now, let $o(Q)$ be infinite. By the residual equation, it is sufficient to prove that $\l{rF} < \phi_{\o(Q)}(0)$ for any tree $rF \in \T(Q)$.
    
    First, we prove this in the case $\o(Q) < \phi_{\o(Q)}(0)$ by induction on $rF$. Note that $\L{rF} \embeds \T\left(\L{r}, 1 + Q \times \S\left(\bigcup_{t \in F} \L{t}\right)\right)$, therefore:
    \begin{align*}
        \l{rF}
        &\leq \o\left(\T\left(\L{r}, 1 + Q \times \S\left(\bigcup_{t \in F} \L{t}\right)\right)\right)\\
        &\leq \overline{\phi}_{\l{r}}\left(1 + \o(Q) \otimes 2^{\bigoplus_{t \in F} \l{t}}\right)\\
        &< \phi_{\o(Q)}(0)
    \end{align*}
    
    Now, assume that $\o(Q) = \phi_{\o(Q)}(0)$. Let $X \subset Q$ be the set of labels of $rF$. Note that $\L{rF} \subset \T\left(\bigcup_{x \in X} \L[Q]{x}\right)$. Indeed, if the tree $t$ has a vertex with label from $\bigcap_{x \in X} L_\geq(x)$, then we can map all vertices of $rF$ to that vertex, so $rF \leq t$. Therefore, by induction, we have 
    $$\l{rF} \leq \o\left(\T\left(\bigcup_{x \in X} \L{x}\right)\right) \leq \phi_{\bigoplus_{x \in X} \l{x}}(0) < \o(Q) = \phi_{\o(Q)}(0).$$
\end{proof}

By \cref{o(S)}, we immediately obtain:

\begin{corollary}
    Let $Q$ be a WQO. If $\o(Q) \geq 2$, then $\o(\F(Q)) \leq \phi_{-1+\o(Q)}(0)$.
\end{corollary}

\section{Width}\label{section w}

In contrast to the maximal order type, the width of $\T(Q)$ and $\F(Q)$ cannot be bounded above in terms of the width of $Q$ alone. Indeed, consider a WQO $Q$ and its minimal element $r$. We have the embedding $F \mapsto rF$ from $\S[2](Q)$ to $\T(Q)$, hence $\w(\T(Q)) \geq \w(\S[2](Q))$. In \cite[example 3.7]{Abriola2023}, it was shown that even in the case $\w(Q) = 2$, the width of $\S(Q)$ can be arbitrarily large. Therefore, the upper bounds of the widths of $\S[2](Q) = \S(Q) \setminus (\{\emptyset\} + Q)$, $\T(Q)$, and $\F(Q)$ do not exist. The main goal of this section is to establish the lower bound.

We will heavily use the following substructure, which is similar to the structure used in the proof of the lower bounds in \cite{Selivanov2025}.
\begin{definition}\label{w rec tree}
    Let $Q$ be an initial segment of a WQO $Q'$, and let $a \in Q' \setminus Q$. For a set of trees $X \subset \T(Q')$, we define $\T(Q, a, X)$ to be the set of those $Q'$-labeled trees that have the following properties:
    \begin{enumerate}
        \item The tree has at least one vertex labeled with $a$.
        \item If all ascendants of some vertex have labels from $Q$, then the vertex has a label from $Q \cup \{a\}$.
        \item If, in addition, that vertex has a label $a$, then it is not a leaf, and all its immediate descendants' subtrees are in $X$.
    \end{enumerate}
\end{definition}

We also define the following subsets of $\T(Q, a, X)$:
\begin{itemize}
    \item The subset $\T^Q(Q, a, X)$ of trees with root label in $Q$.
    \item For $r \in Q$, the subset $\T^r(Q, a, X)$ of trees with root label $r$.
    \item For $r \in Q$, the subset $\T^{\nleq r}(Q, a, X)$ of trees whose root label is not less than or equal to $r$.
\end{itemize}

The next two lemmas show how the widths of these structures relate to each other.

\begin{lemma}\label{restrict the root}
    For any WQO $Q$ and any $r \in Q$, $\w(\T^Q(Q, a, X)) \geq \w(\T^r(Q, a, X))$. Moreover, $\w(\T^Q(Q, a, X)) \geq \w(\T^{\nleq r}(Q, a, X))$ if the latter is at least $3$.
\end{lemma}
\begin{proof}
    The first claim immediately follows from \cref{embedding}. For the second claim, consider a reflection $F \mapsto rF$ from $\S[2](\T^{\nleq r}(Q, a, X))$ to $\T^Q(Q, a, X)$. By \cref{w(S2),embedding}, $\w(\T^Q(Q, a, X)) \geq \w(\S[2](\T^{\nleq r}(Q, a, X))) \geq \w(\T^{\nleq r}(Q, a, X))$.
\end{proof}

\begin{lemma}\label{w(T) is epsilon}
    Let $Q$ be a WQO, $r \in Q$, $\k{r} > 0$. If one of the widths $\w(\T^r(Q, a, x))$ and $\w(\T^{\nleq r}(Q, a, X))$ is at least 4, then they are equal and are epsilon numbers or equal to $\omega$.
\end{lemma}
\begin{proof}
    Let $s \in Q$ be incomparable with $r$. We have the reflections $\S[2](\T^{\nleq r}(Q, a, X)) \embeds \T^r(Q, a, X)$ and $\T^r(Q, a, X) \embeds \T^{\nleq r}(Q, a, X)$ given by $F \mapsto rF$ and $t \mapsto st$, respectively. By \cref{embedding}, $\w(\T^{\nleq r}(Q, a, X)) \geq \w(\T^r(Q, a, X)) \geq \w(\S[2](\T^{\nleq r}(Q, a, X)))$. If $\w(\T^{\nleq r}(Q, a, X))$ is finite, then it is at least 4 and $\w(\S[2](\T^{\nleq r}(Q, a, X))) > \w(\T^{\nleq r}(Q, a, X))$, a contradiction. Hence, it is infinite and $\w(\T^{\nleq r}(Q, a, X)) = \w(\T^r(Q, a, X)) =2^{\w(\T^{\nleq r}(Q, a, X))}$, which proves the claim.
\end{proof}

And the following lemma formalizes the observation that $\T(Q, a, X)$ is ``strictly monotone'' by $X$.

\begin{lemma}\label{w(T) is monotone by X}
    Let $Q$ be a WQO, $r \in Q$, $\k{r} > 0$, and let $a$ be not less than or equal to any element of $Q$.
    Suppose that $X$ is a non-empty set of trees, and for all trees in $X$, the root label is not less than or equal to $a$.
    If $\w(X) \geq 4$, then $\w(\T^{\nleq r}(Q, a, X)) \geq \w(X)$, $\w(\T^r(Q, a, X)) > \w(\T^r(Q, a, \K{x}))$ and $\w(\T^{\nleq r}(Q, a, X)) > \w(\T^{\nleq r}(Q, a, \K{x}))$.
\end{lemma}
\begin{proof}
    The first claim follows from the reflection $x \mapsto ax$ from $X$ to $\T^{\nleq r}(Q, a, X)$.
    It implies that $\w(\T^{\nleq r}(Q, a, X))$ is at least 4, so by \cref{w(T) is epsilon}, the second and the third statements are equivalent.
    
    For $x \in X$, note that $\K[\T^{\nleq r}(Q, a, X)]{a x} \supset \T^{\nleq r}(Q, a, \K{x})$. Indeed, every tree $t$ from the latter set contains a subtree $a y$ with $y \perp x$, hence $t \nleq a x$. Conversely, $x$ is not less than or equal to any subtree of $t$ of the form $a y$, so $a x \nleq t$. Now the third claim follows from \cref{embedding}.
\end{proof}

Now, we define a modification of the Veblen function, which will serve as a lower bound in our intermediate estimates. It is defined when the second argument is at least 4.
\begin{align*}
    f_1(4) &= \omega\\
    f_1(n+5) &= \epsilon_n &&\text{for } n < \omega\\
    f_\alpha(4+\beta) &= \phi_\alpha(\beta) &&\text{for } \alpha > 1 \text{ or } \beta \geq \omega
\end{align*}

We will need the characterization of $f_\alpha(\beta)$ as the fixed point. We give it without proof, since it is an easy exercise on ordinal arithmetic.

\begin{lemma}\label{f def}
    For $\alpha \geq 1, \beta \geq 4$, $f_\alpha(\beta)$ is the least ordinal with the following properties:
    \begin{itemize}
        \item It is an epsilon number or equal to $\omega$.
        \item It is greater than $f_\alpha(\beta')$ for all $4 \leq \beta' < \beta$.
        \item It is a fixed point of $f_{\alpha'}$ for all $1 \leq \alpha' < \alpha$.
    \end{itemize}
\end{lemma}

Now, we are ready to state and prove the bound on $\w(\T(Q))$ via $\o(Q)$.

\begin{lemma}\label{w via o}
    Let $Q$ be a non-empty WQO and let $a$ be not less than or equal to any element of $Q$.
    Let $*$ be incomparable with $a$ and all elements of $Q$.
    Suppose that for all trees in $X$, the root label is not less than or equal to $a$.
    If $\w(X) \geq 4$, then $\w(\T^*(Q\sqcup\{*\}, a, X)) \geq f_{\o(Q)}(\w(X))$.
\end{lemma}
\begin{proof}
    We prove the claim by induction on $\o(Q)$ and $\w(X)$. It is sufficient to check three properties from \cref{f def}.
    The first property is implied by \cref{w(T) is epsilon}, and the second follows from induction and \cref{w(T) is monotone by X}. For the third property, take any $b \in Q$ and note that $\T^*(Q\sqcup\{*\}, a, X) \supset \T^*(\L{b}\sqcup\{*\}, b, \T^*(Q\sqcup\{*\}, a, X))$, hence by induction $\w(\T^*(Q\sqcup\{*\}, a, X)) \geq \w(\T^*(\L{b}\sqcup\{*\}, b, \T^*(Q\sqcup\{*\}, a, X))) \geq f_{\l{b}}(\w(\T^*(Q\sqcup\{*\}, a, X)))$, so we are done.
\end{proof}

\begin{proposition}\label{w T(Q cup *)}
    For a WQO $Q$ of width at least 2, $\w(\T(Q \sqcup \{*\})) \geq \phi_{\o(Q)}(0)$.
\end{proposition}
\begin{proof}
    Let $\T^*(Q\sqcup\{*\})$ be the subset of those trees whose root is labeled with $*$. If $x,y \in Q$ are two incomparable elements, then trees $*x*x,*x*y,*y*x,*y*y$ are pairwise incomparable, hence $\w(\T^*(Q\sqcup\{*\})) \geq 4$. Therefore, it is sufficient to prove that $\w(\T^*(Q\sqcup\{*\}))$ is a fixed point of $f_\alpha$ for all $1 \leq \alpha < \o(Q)$.
    Indeed, choose $a \in Q$ such that $\l{a} \geq \alpha$. Then we have $\T^*(Q\sqcup\{*\}) \supset \T^*(\L{a}\sqcup\{*\}, a, \T^*(Q\sqcup\{*\}))$, hence, by \cref{embedding,w via o}, $\w(\T^*(Q\sqcup\{*\})) \geq \w(\T^*(\L{a}\sqcup\{*\}, a, \T^*(Q\sqcup\{*\}))) \geq f_{\l{a}}(\w(\T^*(Q\sqcup\{*\}))) \geq f_\alpha(\w(\T^*(Q\sqcup\{*\})))$, and we are done.
\end{proof}

If we combine this proposition with the fact that $\o(Q) \geq \w(Q)$ and apply the residual equation for $\w(Q)$, we may get $\o(\T(Q)) \geq \w(\T(Q)) \geq \phi_{\w(Q)-1}(0)$. This bound meets the upper bound only in the case $\w(Q) = \o(Q) = \lambda$, where $\lambda$ is a limit ordinal. In the rest of the section, we improve the bound to get $\w(\T(Q)) \geq \phi_{\w(Q)}(0)$, which meets the upper bound in the case $\w(Q) = \o(Q)$ (not necessarily a limit ordinal).

First, we prove an analog of the \cref{w via o} for $\w(Q)$.

\begin{lemma}
    Let $Q$ be a WQO whose width is a limit ordinal, and let $a$ be not less than or equal to any element of $Q$.
    Suppose that for all trees in $X$, the root label is not less than or equal to $a$.
    If $\w(X) \geq 4$, then $\w(\T^Q(Q, a, X)) \geq f_{\w(Q)}(\w(X))$.
\end{lemma}
\begin{proof}
    By induction on $\w(X)$. Without loss of generality, we may assume that $\w(\T^Q(Q, a, X))$ is the least possible among all pairs of $Q$ and $X$ of the widths $\w(Q)$ and $\w(X)$, respectively. Therefore, for any subset $Q' \subset Q$ of the width $\w(Q)$, we have $\w(\T^{Q'}(Q', a, X)) = \w(\T^Q(Q, a, X))$. We may also assume that $\h(Q)$ is the least possible among all $Q$ for which previous assumptions hold, so $Q$ does not contain a subset of the smaller height but the same width.

    We will prove the claim using \cref{f def}. First, let $r \in Q$ be minimal such that $\k{r} > 0$. Since $\w(L_\leq(r)) = 1$ and $\w(Q)$ is a limit ordinal, by \cref{h o w(cup)}, $\w(L_\nleq(r)) = \w(Q)$. Therefore, $\w(\T^Q(Q, a, X)) = \w(\T^{L_\nleq(r)}(L_\nleq(r), a, X)) = \w(\T^{\nleq(r)}(Q, a, X))$, and by \cref{w(T) is epsilon,w(T) is monotone by X}, the first two conditions from \cref{f def} are satisfied.

    Now, we prove that $\w(\T^Q(Q, a, X))$ is closed under $f_\alpha$ for any $\alpha < \w(Q)$. Let $r, * \in Q$ be such that $r \perp *$ and $\k{r, *} \geq \alpha$. Consider the set $S$ of those $s \in Q$ for which there exists a sequence $r_0 = r, r_1, \dots, r_{n-1},r_n = s$ such that $r_i \nleq r_{i-1}$ for all $1 \leq i \leq n$. By construction, for any $s \in Q \setminus S$, we have $\forall s' \in S (s \leq s')$, hence $Q = (Q \setminus S) + S$. By assumption, $\w(Q \setminus S) < \w(Q)$, so by \cref{h o w(cup)}, $\w(S) = \w(Q)$. $S$ is a union of upwards-closed sets $L_{\not<,S}(s)$, hence it is a union of finite amount of them, $S = S' \cup \{s_1, \dots, s_m\}$ where $S' = L_{\nleq,S}(s_1) \cup \dots \cup L_{\nleq,S}(s_m)$. By \cref{h o w(cup)}, $\w(S') = \w(S)$ because $\w(S)$ is a limit ordinal.

    $\T^{S'}(S', a, X) = \T^{\nleq s_1}(S', a, X) \cup \dots \cup \T^{\nleq s_m}(S', a, X)$. Since $\w(\T^{S'}(S', a, X)) = \w(\T^Q(Q, a, X))$ is an epsilon number or $\omega$, it is additively indecomposable. Therefore, for some $i$ we have $\w(\T^{\nleq s_i}(S', a, X)) = \w(\T^Q(Q, a, X))$. By construction of $S$, there exists a sequence $r_0 = r, r_1, \dots, r_{n-1}, r_n = s_i$ such that $r_i \nleq r_{i-1}$. By \cref{w(T) is epsilon}, $\w(\T^Q(Q, a, X)) = \w(\T^{\nleq s_i}(S', a, X)) = \w(\T^{r_n}(S', a, X)) \leq \w(\T^{\nleq r_{n-1}}(S', a, X)) = \w(\T^{r_{n-1}}(S', a, X)) \leq \dots \leq \w(\T^{\nleq r_0}(S', a, X))$. Observe that $\T^Q(Q, a, X) \supset \T^*(\K{r, *} \sqcup \{*\}, r, \T^{\nleq r}(S', a, X))$, hence by \cref{w via o}, we have $\w(\T^Q(Q, a, X)) \geq f_{\o(\K{r, *})}(\w(\T^{\nleq r}(S', a, X))) \geq f_\alpha(\w(\T^Q(Q, a, X)))$, which completes the proof.
\end{proof}

The next lemma extends the result on non-limit $\w(Q)$.

\begin{lemma}\label{w via w}
    Let $Q$ be a WQO of the infinite width and let $a$ be not less than or equal to any element of $Q$.
    Suppose that for all trees in $X$, the root label is not less than or equal to $a$.
    If $\w(X) \geq 4$, then $\w(\T^Q(Q, a, X)) \geq f_{\w(Q)}(\w(X))$.
\end{lemma}
\begin{proof}
    We prove the claim by induction on $\w(Q)$. The case of a limit $\w(Q)$ is verified by the previous lemma, so we may assume that $\w(Q)$ is a successor ordinal.
    
    Choose $r \in Q$ such that $\k{r} = \w(Q)-1$. The statement will follow from \cref{restrict the root} and the fact that $\w(\T^{\nleq r}(Q, a, Y)) \geq f_{\w(Q)}(\w(Y))$ for $\w(Y) \geq 4$, which we prove by induction on $\w(Y)$. Again, we check three properties from \cref{f def}. The first two properties follow from \cref{w(T) is epsilon,w(T) is monotone by X}. The third property follows by induction from the fact that $\T^{\nleq r}(Q, a, Y) \supset \T^{\K{r}}(\K{r}, r, \T^{\nleq r}(Q, a, Y))$.
\end{proof}

Finally, we apply this lemma to get the desired bound.

\begin{theorem}
    For a WQO $Q$ of the width at least 3, $\w(\T(Q)) \geq \phi_{-1+\w(Q)}(0)$.
\end{theorem}
\begin{proof}
    First, we assume that $\w(Q)$ is a successor ordinal. Let $r \in Q$ be such that $\k{r} = \w(Q)-1$. If $\w(Q)$ is finite, then by \cref{w T(Q cup *)}, $\w(\T(Q)) \geq \w(\T(\K{r}\sqcup\{r\})) \geq \phi_{\o(\K{r})}(0) \geq \phi_{\w(Q)-1}(0)$ and we are done. 
    
    So, we assume that $\w(Q)$ is infinite but not limit. Let $\T^{\bot r}(Q)$ be the subset of those trees whose root is incomparable with $r$, and let $s \in \K{r}$ be such that $\k{s} \geq 2$. By \cref{w T(Q cup *)}, $\w(\T^{\bot r}(Q)) \geq \w(\T(\K{r,s}\sqcup\{s\})) \geq \phi_{\o(\K{r,s})}(0) \geq 4$. Observe that $\T^{\bot r}(Q) \supset \T^{\K{r}}(\K{r}, r, \T^{\bot r}(Q))$. By \cref{embedding,w via w}, $\w(\T^{\K{r}}(Q)) \geq \w(\T^{\K{r}}(\K{r}, r, \T^{\K{r}}(Q))) \geq f_{\k{r}}(\w(\T^{\bot r}(Q)))$. Hence $\w(\T^{\bot r}(Q))$ is a fixed point of $f_{\w(Q)-1}$, so $\w(\T(Q)) \geq \w(\T^{\bot r}(Q)) \geq \phi_{\w(Q)}(0)$.

    Now, let $\w(Q)$ be a limit ordinal. For any $\alpha < \w(Q)$, there exists $Q' \subset Q$ of the width $\alpha+1$, hence $\w(\T(Q)) \geq \w(\T(Q')) \geq \phi_{\alpha+1}(0)$, so $\w(\T(Q)) \geq \sup_{\alpha<\w(Q)} \phi_{\alpha+1}(0) = \phi_{\w(Q)}(0)$.
\end{proof}

\begin{corollary}
    For a WQO $Q$ of the width at least 3, $\w(\F(Q)) \geq \phi_{-1+\w(Q)}(0)$.
\end{corollary}

\begin{corollary}
    Let $Q$ be a WQO such that $\o(Q) = \w(Q) = \alpha \geq 3$. Then $\o(\T(Q)) = \w(\T(Q)) = \o(\F(Q)) = \w(\F(Q)) = \phi_{-1+\alpha}(0)$.
\end{corollary}

Since for any ordinal $\alpha$ there is a WQO $Q$ with $\o(Q) = \w(Q) = \alpha$, the given bounds are tight.

We may apply the last corollary iteratively. Thus, let $\T_m(Q)$ denote the \textit{iterated $h$-preorder} (introduced and studied in \cite{Selivanov2011}): $\T_0(Q) = Q, \T_{m+1}(Q) = \T(\T_m(Q))$. $\F_m(Q)$ is a forest variant of $\T_m(Q)$.

\begin{corollary}
    Let $Q$ be a WQO such that $\o(Q) = \w(Q) = \alpha \geq 3$. Then $\o(\T_m(Q)) = \w(\T_m(Q)) = \o(\F_m(Q)) = \w(\F_m(Q)) = \phi_{\phi_{\dots_{\phi_\alpha(0)}}(0)}(0)$ (with $m$ applications of $\phi$).
\end{corollary}

In particular case when $Q$ is a finite antichain $\overline{k}$, $\T_m\left(\overline{k}\right)$ is an initial segment of $\T_{m+1}\left(\overline{k}\right)$. Hence we can take a limit to obtain the invariants of the $\omega$-th iteration $\T_\omega\left(\overline{k}\right)$ and $\F_\omega\left(\overline{k}\right)$. We do not define it formally; for technical details and additional information for the iterated h-preorder see e.g. \cite{Selivanov2011,Selivanov2022,Alaev2025}.

\begin{corollary}
    For any $k \geq 3$, $\o(\T_\omega\left(\overline{k}\right)) = \w(\T_\omega\left(\overline{k}\right)) = \o(\F_\omega\left(\overline{k}\right)) = \w(\F_\omega\left(\overline{k}\right)) = \Gamma_0$, where $\Gamma_0$ is a Feferman–Schütte ordinal -- the least $\alpha$ such that $\phi_\alpha(0) = \alpha$.
\end{corollary}

\section{Future work}

In this work, we proved the bounds for the maximal order type and the width, but not for the height. This invariant will be considered in the upcoming work by the present author and Victor Selivanov.

Our bounds work well in the case when the maximal order type and the width of the base structure coincide, but in the other cases, the exact values of ordinal invariants are unknown. These questions are still open even in the finite case; we do not know the answer even for the structure $(\mathsf{P}(k), \subseteq)$, where $k \geq 3$ ($\mathsf{P}$ denotes the powerset).

Another possible direction for future work is the ordinal invariants of $h$-preorder on \textit{infinite} well-founded trees \cite{Selivanov2007}. The presented methods do not work directly in this case, and the main obstacle is the absence of the estimates on the maximal order type and the width of the powerset. Even in the cases when the powerset is a WQO, its ordinal invariants cannot be computed the same way as in the finitary case, since the residual sets have much less straightforward structure.

\bibliographystyle{plain}
\bibliography{refs}

\end{document}